\documentclass[oneside,reqno,12pt]{amsart}

\usepackage{amsmath,amssymb,amsfonts}
\usepackage{kpfonts}
\usepackage{mathtools}
\usepackage{bbm}
\usepackage{hyperref}
\usepackage[foot]{amsaddr}

\usepackage{bbm}

\newtheorem{theorem}{Theorem}
\newtheorem{lemma}{Lemma}

\theoremstyle{definition}
\newtheorem{definition}{Definition}

\theoremstyle{remark}

\newcommand{\ii}{\ensuremath{\mathbf{1}}}
\newcommand{\rr}{\ensuremath{\mathbb{R}}}

\newcommand{\xx}{{\sf{X}}}
\newcommand{\pp}{{\sf{P}}}
\newcommand{\llf}{{\sf{L}}}
\newcommand{\ff}{\mathcal{F}}

\newcommand{\zz}{\mathbb{Z}}

\newtheorem{asswx}{Assumption~A\hspace*{-2pt}}
\begin{document}

\title[Wave equation driven by a stochastic measure]{Wave equation with unbounded terms driven by a stochastic measure}

\author{Mykhailo Debeliak$^1$}

\author{Vadym Radchenko$^1$}

\address
{$^1$Department of Mathematical Analysis, Taras Shevchenko National University of Kyiv, 01601 Kyiv, Ukraine. Emails: mishadebelyak@knu.ua, vadymradchenko@knu.ua}

\begin{abstract}
The one-dimensional wave equation with additive noise driven by a stochastic measure is considered. For the stochastic integrator, we assume only the sigma-additivity in probability and that $|y|^\tau$ is integrable on $\rr$ for some $\tau>1/2$. The class of stochastic measures includes many important cases of random processes. In the equation, the stochastic measure depends only on the spatial variable. The existence, uniqueness, and H\"{o}lder regularity of the mild solution are proved. In the previous publications, similar results were obtained for the case of bounded terms of the wave equation driven by a stochastic measure. In our paper, one term may be unbounded, satisfying the linear growth condition. Nevertheless, we prove that the solution is bounded in this case.
\end{abstract}


\maketitle

\section*{Introduction} 
In this paper, we consider the stochastic equation, which can formally be written as
\begin{equation}
	\label{weq} \left\{
	\begin{array}{l}
		\dfrac{\partial^2 u(t,x)}{\partial t^2}
		=\dfrac{\partial^2 u(t,x)}{\partial x^2}
		+f(t,x,u(t,x))+ \sigma(t,x)\,\dot{\mu}(x) ,\\
		\\
		u(0,x)=u_0(x),\quad \dfrac{\partial u(0,x)}{\partial t}=v_0(x),
	\end{array}
	\right.
\end{equation}
where $(t,x)\in [0,T]\times{\rr}$, $\mu$ is an stochastic measure (see Definition~\ref{def:stme}) defined on the Borel $\sigma$-algebra~$\rr$.

We consider equation~\eqref{weq} in the mild sense (see \eqref{eqwavex} below).

The main purpose of the paper is to prove the existence, uniqueness, and boundedness of the solution in the case of possibly unbounded $f$. In addition, we obtain the H\"{o}lder regularity of~$u$.

Note that the class of stochastic measures includes many classes of stochastic integrators (for some examples, see Section~\ref{sc:prelim}).

For bounded~$f$, the solution of equation~\eqref{weq} was studied in \cite{bodnarchuk10}, \cite[Section 3.3]{radbook}. A similar equation with a stochastic measure dependent on the time variable was considered in~\cite{bodnarchuk2017t}. The wave equation on the plane was studied in~\cite{bodrad19}.

Other classes of equations driven by stochastic measures were studied in~\cite{radbook}, \cite{manikin_2025}.

In publications on stochastic partial differential equations, usually special classes of integrators are considered. For instance, equations driven by the L\'{e}vy processes are studied in~\cite{peszab}, and those driven by the Wiener process are studied in~\cite{dalbook}. In our paper, we consider a relatively wide class of integrators.

The rest of the paper is organized as follows. In Section~\ref{sc:prelim}, we give the basic information on integrals with respect to stochastic measures. Section~\ref{sc:assmr} contains the formulation of the problem, our assumptions, and the statement of the main theorem. In Section~\ref{sc:auxl} we give some auxiliary lemmas, and in Section~\ref{sc:proff} we prove the main result.

\section{Preliminaries}\label{sc:prelim}

Let $\llf_0=\llf_0(\Omega, {\ff}, {\pp} )$ be the set of all real-valued random variables defined on the complete probability space $(\Omega, {\ff}, {\pp} )$. Convergence in $\llf_0$ means the convergence in probability. Let ${\xx}$ be an arbitrary set and ${\mathcal{B}}$ a $\sigma$-algebra of subsets of ${\xx}$.

\begin{definition}\label{def:stme}
	A $\sigma$-additive mapping $\mu:\ {\mathcal{B}}\to \llf_0$ is called {\em stochastic measure} (SM).
\end{definition}

We do not assume the existence of moments or martingale properties for SM.

For deterministic measurable functions $g:\xx\to\rr$, an integral of the form $\int_{\xx}g\,d\mu$ is studied
in~\cite[Chapter 1]{radbook} (see also~\cite[Chapter 7]{kwawoy}). In particular, every bounded measurable $g$ is integrable w.~r.~t. any SM~$\mu$.

Important examples of SMs are orthogonal stochastic measures, $\alpha$-stable random measures defined on a $\sigma$-algebra for $\alpha\in (0,1)\cup(1,2]$ (see \cite[Chapter 3]{samtaq}). Many examples of the SMs on the Borel $\sigma$-algebra may be given by the Wiener-type integral
\begin{equation}\label{eq:muax}
	\mu(A)=\int_{(0,1]} {\ii}_A(t)\,{d}X_t.
\end{equation}

We note the following important cases of the process $X_t$ in~\eqref{eq:muax} that generate SM.

\begin{enumerate}
	
	\item\label{itmart} $X_t$~-- any square integrable martingale.
	
	\item\label{itfrbr} $X_t=W_t^H$~-- the fractional Brownian motion with Hurst index $H>1/2$, see Theorem~1.1~\cite{memiva}.
	
\end{enumerate}

Other examples and basic properties of SMs may be found in \cite[Section~1.2]{radbook}.

For the convergence of the series of integrals w. r. t. $\mu$, we will use the following statement.

\begin{lemma} \label{lmfkmu}  (Lemma~3.1~\cite{radt09}, Lemma 2.3~\cite{radbook}) Let $\mu$ be an SM on $(\xx,{\mathcal{B}})$, $f_l:\ {\xx}\to \rr,\ l\ge 1$, be measurable functions such that
	$\hat{f}(x)=\sum_{l=1}^{\infty} |{f_l}(x)|$ is integrable w.~r.~t.~$\mu$. Then
	\begin{equation} \label{eqfkmu}
		\sum_{l=1}^{\infty}\Bigl(\int_{\xx} f_l\,d\mu \Bigr)^2<\infty\quad
		\mbox{\textrm ~a.~s.}
	\end{equation}
\end{lemma}

Further, we consider an SM $\mu$ defined on the Borel $\sigma$-algebra in $\rr$.

For all $n\ge 0$, put
\[
d_{kn}=k 2^{-n},\quad 0\le k\le 2^n,\quad \Delta_{kn}=(d_{(k-1)n}, d_{kn}],\quad
1\le k\le 2^n\, .
\]

The following statement is a key tool for estimating the stochastic integral.
\begin{lemma}\label{lm:essih} (Lemma~3~\cite{radt15}, Lemma 2.8~\cite{radbook})
	Let $Z$ be an arbitrary set, and the function $q(z,x):
	Z\times [0, 1]\to\rr$ be such that all paths $q(z,\cdot)$ are continuous on $[0,1]$. Denote
	\[
	q_n(z,x)=\sum_{1\le k\le 2^n}q(z,{d}_{(k-1)n}){ \ii}_{\Delta_{kn}}(x).
	\]
	Then the random function
	\[
	\eta(z)=\int_{(0,1]}q(z, x)\,d\mu(x),\ z\in Z,
	\]
	has a version
	\begin{equation}\label{eqvers}
		\begin{split}
			\widetilde{\eta}(z)  = \int_{(0,1]}q_0(z,x)\,d\mu(x)
			+\sum_{n\ge 1}\Bigl(\int_{(0, 1]}q_n(z,x)\,d\mu(x)-\int_{(0, 1]}q_{n-1}(z,x)\,d\mu(x)\Bigr)
		\end{split}
	\end{equation}
	such that for all $\beta>0$, $\omega\in\Omega$, $z\in Z$,
	\begin{equation}\label{eq:esqm}
		\begin{split}
			|\widetilde{\eta}(z)| \le  |q(z,0)\mu((0, 1])|\\
			+ \Bigl\{\sum_{n\ge 1}2^{n\beta}
			\sum_{1\le k\le 2^{n}}|q(z,d_{kn})-q(z,d_{(k-1)n})|^2\Bigr\}^{1/2}
			\Bigl\{\sum_{n\ge 1}2^{-n\beta}\sum_{1\le k\le 2^{n}}|\mu(\Delta_{kn})|^2\Bigr\}^{1/2}.
		\end{split}
	\end{equation}
\end{lemma}

Note that for $\beta>0$,
\begin{eqnarray*}
	\sum_{n\ge 1} 2^{-n\beta} \sum_{1\le k\le 2^{n}} |\mu(\Delta_{kn})|^2 < +\infty\ {\rm a.~s.}
\end{eqnarray*}
This follows from Lemma~\ref{lmfkmu}, applied for
\[
\{f_l (y),\ l\ge 1\} = \{ 2^{-n\beta/2}\, \ii_{\Delta_{kn}}(y),\ n\ge 1,\ 1\le k\le 2^{n}\},
\]
(see also Lemma 2.4~\cite{radbook}).

In the sequel, by $C$ and $C(\omega)$ we will denote positive constants and positive random constants, respectively, whose exact values are not important.

\section{Assumptions and main result}\label{sc:assmr}

The solution of equation~\eqref{weq} is understood in a mild sense,
\begin{equation}\label{eqwavex}
	\begin{split}
		u(t,x)
		= \frac{1}{2}(u_0(x+t)+u_0(x-t))+ \frac{1}{2}\int_{x-t}^{x+t} v_0(y)\,dy\\
		+ \frac{1}{2}\int_0^t\,ds\int_{x-(t-s)}^{x+(t-s)}  f(s,y, u(s,y))\,dy
		+ \frac{1}{2}\int_{|y-x|<t}\,d\mu(y) \int_{0}^{t-|y-x|}\sigma(s,y)\,ds\, .
	\end{split}
\end{equation}

We impose the following assumptions.

\begin{asswx}\label{asswxu} 
	$u_0(y)=u_0(y, \omega):{\rr}\times\Omega\to\rr$ is measurable, and bounded:
	\[
	|u_0(y)|\le C_{u_0}(\omega)
	\]
	for some random constant $C_{u_0}(\omega)$.
\end{asswx}

\begin{asswx}\label{asswxuh}
	$u_0(y)$ is H\"{o}lder continuous, i.~e. for some $\beta(u_0)>0$ \[
	|u_0(y_1)-u_0(y_2)|\le L_{u_0}(\omega) |y_1-y_2|^{\beta(u_0)}\]
	for some random constant $L_{u_0}(\omega)$.
\end{asswx}

\begin{asswx}\label{asswxv}
	$v_0(y)=v_0(y, \omega):{\rr}\times\Omega\to\rr$ is measurable, and bounded:
	\[
	|v_0(y)|\le C_{v_0}(\omega)
	\]
	for some random constant $C_{u_0}(\omega)$.
\end{asswx}

\begin{asswx}\label{asswxfb} $f(s, y, v):[0,T]\times{\rr}\times\rr\to\rr$ is measurable, and satisfies the following linear growth condition in $v$:
	\[
	|f(s,y,v)|\le C_{f}(1+|v|),\quad v\in\rr,
	\]
	where $C_f$ is a constant.
\end{asswx}

\begin{asswx}\label{asswxfl} $f$ is locally Lipschitz continuous in $v$:
	\[
	|f(s, y, v_1)-f(s, y, v_2)|\le L_{f,K} |v_1-v_2|,\quad  |v_1|, |v_2|\le K
	\]
	for some constant $L_{f,K}$.
\end{asswx}

\begin{asswx}\label{asswxs} 
	$\sigma(s,y):[0, T]\times{\rr}\to\rr$ is measurable, and bounded:
	\[
	|\sigma(s,y)|\le C_{\sigma}
	\]
	for some constant $C_{\sigma}$.
\end{asswx}

\begin{asswx}\label{asswxsh}
	$\sigma(s,y)$ is H\"{o}lder continuous, and for some $\beta(\sigma)>1/2$
	\[
	|\sigma(s_1,y_1)-\sigma(s_2,y_2)|\le L_{\sigma}(|s_1-s_2|^{\beta(\sigma)}+|y_1-y_2|^{\beta(\sigma)})
	\]
	where $L_{\sigma}$ is a constant .
\end{asswx}

\begin{asswx}\label{asswxm}
	The function $|y|^{\tau}$ is integrable w.~r.~t.~$\mu$ on $\rr$ for some $\tau>1/2$.
\end{asswx}

Theorem~3.3~\cite{radbook} gives the existence, uniqueness, and the H\"{o}lder regularity of the solution to~\eqref{eqwavex} in the case of bounded $f$.

The main result of our paper is the following.

\begin{theorem}\label{thwavex}
	1) Let Assumptions A\ref{asswxu},  A\ref{asswxv}, A\ref{asswxfb}, A\ref{asswxfl}, and A\ref{asswxs} hold. Then the equation~\eqref{eqwavex} has a solution~$u(t,x)$ that is bounded a.s., i.e.
	\[
	\pp\Bigl\{\sup_{t\in[0,T],x\in\mathbb{R}}|u(t,x)|\le\widetilde{M}(\omega)\Bigr\}=1
	\]
	for some random constant $\widetilde{M}(\omega)$.
	If $\hat{u}(t,x)$ is another bounded solution to~\eqref{eqwavex}, then for each $(t,x)$ we have $u(t,x)=\hat{u}(t,x)$~a.~ s.

	2) In addition, assume that A\ref{asswxuh} and A\ref{asswxsh} hold. Then there exists a version of $u(t,x)$ such that for each
	\[
	K>0,\quad 0\le \gamma\le \beta(u_0),\quad \gamma<\dfrac{1}{2}
	\]
	for some $\widetilde{\Omega}\subset \Omega$, $\pp(\widetilde{\Omega})=1$,
	a random constant $L_{{u},K}(\omega)$, and for each $\omega\in \widetilde{\Omega}$
	it holds
	\[
	|{u}(t_1,x_1)-{u}(t_2,x_2)|\le L_{{u},K}(\omega)(|t_1-t_2|^{\gamma}+|x_1-x_2|^{\gamma})
	\]
	for all $x_1,\ x_2\in[-K,K]$ and $t_1,\ t_2\in [0,T]$.
	($\widetilde{\Omega}$ depends only on~$\mu$, $\sigma$, $K$, and $\gamma$, $L_{{u},K}(\omega)$ depends also on~$u_0$, $v_0$ and $f$).
\end{theorem}

\section{Auxilliary lemmas}\label{sc:auxl}

We will study the stochastic term of~\eqref{eqwavex}:
$$
\vartheta(t,x)=\int_{|y-x|<t}d\mu(y)\int_{0}^{t-|y-x|}\sigma(s,y)\,ds.
$$

In the following statement, we combine Lemmas~3.6 and 3.7~\cite{radbook}.

\begin{lemma}\label{lm:vhold}
	Let Assumptions  A\ref{asswxs} and A\ref{asswxsh} hold. Then for each $K>0$, $0<\gamma<\frac{1}{2}$ there exist $\widehat{\Omega}\subset \Omega$, $\pp(\widehat{\Omega})=1$, and a random constant $L_{\vartheta,K}(\omega)$ such that for any fixed $t\in (0,T]$ for version~\eqref{eqvers} of the process $\vartheta(t,x)$ holds
	\begin{equation*}
		|\vartheta(t_1,x_1)-\vartheta(t_2,x_2)|\le L_{\vartheta,K}(\omega)(|t_1-t_2|^{\gamma}+|x_1-x_2|^{\gamma}),\quad |x_i|\le K,\quad t_i\in (0,T] 
	\end{equation*}
	for all $\omega\in \widehat{\Omega}$ ($\widehat{\Omega}$ and $L_{\vartheta,K}(\omega)$ depend only on~$\sigma$, $\mu$, $K$, and $\gamma$).
\end{lemma}

We will prove the following statement about the boundness of the stochastic term.

\begin{lemma}\label{lm:vbound} Let Assumptions A\ref{asswxs} and A\ref{asswxsh} hold. Then  for version~\eqref{eqvers} of the process $\vartheta(t,x)$ for some $\bar{\Omega}\subset \Omega$, $\pp(\bar{\Omega})=1$,
	a random constant $C_{\vartheta}(\omega)$, and for all $x\in\rr$, $t\in [0,T]$ and $\omega\in \bar{\Omega}$
	holds
	\begin{equation}\label{eq:vbound}
		|\vartheta(t,x)|\le C_{\vartheta}(\omega)
	\end{equation}
	($\bar{\Omega}$ and $C_{\vartheta}(\omega)$ depend only on~$\mu$ and $\sigma$).
\end{lemma}

\begin{proof}
	We have
	\begin{equation} \label{eq:3}
		\vartheta(t,x)=\sum_{j\in\mathbb{Z}}\int_{(j,j+1]}q(z,y)d\mu(y),
	\end{equation}
	where
	$$q(z,y)=1_{|y-x|<t}\int_{0}^{t-|y-x|}\sigma(s,y)ds, \quad z=(t,x).$$
	
	Assumption A\ref{asswxs} implies that
	\begin{equation} \label{eq:4}
		|q(z,y)|\le C_{\sigma}t\le C_{\sigma}T.
	\end{equation}
	
	We will obtain the H\"{o}lder continuity of~$q(z,y)$. The function $q(z,y)$ is continuous in $y$; therefore, Lemma~\ref{lm:essih} may be applied in this case. We consider two cases for $h$.
	
	For the case $0<h<2t$, we have
	\begin{eqnarray*}
		q(z,y+h)-q(z,y) = 1_{|y+h-x|<t}\int_{0}^{t-|y+h-x|}\sigma(s,y+h)ds-1_{|y-x|<t}\int_{0}^{t-|y-x|}\sigma(s,y)ds \\
		= 1_{x-t-h<y\le x-t}\int_{0}^{t-|y+h-x|}\sigma(s,y+h)ds 
		+ 1_{x-t<y\le x+t-h}\int_{0}^{t-|y+h-x|}(\sigma(s,y+h)-\sigma(s))ds\\
		- 1_{x-t<y\le x+t-h}\int_{t-|y+h-x|}^{t-|y-x|}\sigma(s,y)ds
		- 1_{x+t-h<y<x+t}\int_{0}^{t-|y-x|}\sigma(s,y)ds:=B_1+B_2-B_3-B_4.
	\end{eqnarray*}
	In the terms $B_1$, $B_3$, and $B_4$, for each $y$, the lengths of the integration segments are less than $h$, and these terms do not exceed $C_\sigma h$.  By~A\ref{asswxsh}, $|B_2|\le T L_{\sigma} h^{\beta(\sigma)}$. Therefore, 
	\begin{equation}\label{eq:holdq}
		|q(z,y+h)-q(z,y)|\le Ch^{\beta(\sigma)}.    
	\end{equation}
	If $h\ge 2t$, we get
	$$|q(z,y+h)-q(z,y)|\le|q(z,y+h)|+|q(z,y)|\stackrel{\eqref{eq:4}}{\le} 2C_{\sigma}t\le C_{\sigma}h,
	$$
	and~\eqref{eq:holdq} holds in this case too.
	
	Further, to estimate $\vartheta$, 
	we will use equality~\eqref{eq:3}, and apply Lemma~\ref{lm:essih} on each segment $(j,j+1]$. For all $n\ge 0$, $j\in\zz$, denote
	\[
	d_{kn}^{(j)}=j+k 2^{-n},\quad 0\le k\le 2^n,\quad \Delta_{kn}^{(j)}=(d_{(k-1)n}^{(j)}, d_{kn}^{(j)}],\quad
	1\le k\le 2^n\, .
	\]
	For any $0<\beta<2\beta(\sigma)-1$, $\tau>\frac{1}{2}$, for version~\eqref{eqvers} of $\vartheta$ in each segment $(j,j+1]$, we obtain
	\begin{eqnarray*}
		|{\vartheta}(t,x)| \stackrel{\eqref{eq:esqm}}{\le} 
		\sum_{j\in\zz}|q(z,j)\mu((j, j+1])|\\
		+ \sum_{j\in\zz}\Bigl\{\sum_{n\ge 1}2^{n\beta}
		\sum_{1\le k\le 2^{n}}|q(z,d_{kn}^{(j)})-q(z,d_{(k-1)n}^{(j)})|^2\Bigr\}^{1/2}
		\Bigl\{\sum_{n\ge 1}2^{-n\beta}\sum_{1\le k\le 2^{n}}|\mu(\Delta_{kn}^{(j)})|^2\Bigr\}^{1/2}\\
		\le \Bigl(\sum_{j\in\mathbb{Z}}|q(z,j)|^{2}(1+|j|)^{-2\tau}\Bigr)^{1/2}\Bigl(\sum_{j\in\mathbb{Z}}(1+|j|)^{2\tau}|\mu((j,j+1])|^{2}\Bigr)^{1/2} \\
		+\Bigl(\sum_{j\in\mathbb{Z}}(1+|j|)^{-2\tau}\sum_{n\ge1}2^{n\beta}\sum_{k=1}^{2^{n}}|q(z,d_{kn}^{(j)})-q(z,d_{(k-1)n}^{(j)})|^{2}\Bigr)^{1/2}\\
		\times\Bigl(\sum_{j\in\mathbb{Z}}(1+|j|)^{2\tau}\sum_{n\ge1}2^{-n\beta}\sum_{k=1}^{2^{n}}|\mu(\Delta_{kn}^{(j)})|^{2}\Bigr)^{1/2}\\
		\stackrel{\eqref{eq:4},\eqref{eq:holdq}}{\le} C\Bigl(\sum_{j\in\mathbb{Z}}(1+|j|)^{-2\tau}\Bigr)^{1/2}\Bigl(\sum_{j\in\mathbb{Z}}(1+|j|)^{2\tau}|\mu((j,j+1])|^{2}\Bigr)^{1/2} \\
		+\Bigl(\sum_{j\in\mathbb{Z}}(1+|j|)^{-2\tau}\sum_{n\ge1}2^{n\beta}\sum_{k=1}^{2^{n}}2^{-2n\beta(\sigma)}\Bigr)^{1/2}
		\Bigl(\sum_{j\in\mathbb{Z}}(1+|j|)^{2\tau}\sum_{n\ge1}2^{-n\beta}\sum_{k=1}^{2^{n}}|\mu(\Delta_{kn}^{(j)})|^{2}\Bigr)^{1/2}.
	\end{eqnarray*}
	Here, the series with $(1+|j|)^{-2\tau}$ are finite because $\tau>1/2$, and the sums with $\mu$ have the form
	$\sum_{l=1}^{\infty}\Bigl(\int_{\rr} f_l\,d\mu \Bigr)^2$, where
	\begin{eqnarray*}
		\{f_l (y),\ l\ge 1\} = \{(|j|+1)^{\tau}\, \ii_{(j, j+1]}(y),\ j\in \zz\},\\
		\{f_l (y),\ l\ge 1\} = \{(|j|+1)^{\tau} 2^{-n\beta}
		\ii_{\Delta_{kn}^{(j)}}(y),\ j\in \zz,\ n\ge 1,\ 1\le k\le 2^{n}\}.
	\end{eqnarray*}
	We have $\sum_{l=1}^{\infty}|f_l(y)|\le C(|y|+1)^{\tau}$.
	From A\ref{asswxm} and Lemma~\ref{lmfkmu} it follows that
	\[
	\sum_{l=1}^{\infty}\Bigl(\int_{\xx} f_l\,d\mu \Bigr)^2<+\infty\ \textrm{a.\ s.}
	\]
	We include in $\bar{\Omega}$ all $\omega\in\Omega$ such that this series is finite for both sequences $\{f_l\}$. Thus, $\pp(\bar{\Omega})=1$ and the boundedness of~${\vartheta}$ holds. The set $\bar{\Omega}$ depends only on $\mu$, $\sigma$, and $\beta$.
\end{proof}

\section{Proof of the main result}\label{sc:proff}

In this section, we give the proof of Theorem~\ref{thwavex}.
\begin{proof} 1) For $\vartheta$ we use the version~\eqref{eqvers}, thus Lemmas~\ref{lm:vhold} and \ref{lm:vbound} hold. All considerations are carried out for a fixed $\omega\in\widetilde{\Omega}=\widehat{\Omega}\cap\bar{\Omega}$.
	
	We will use the standard iteration process. Set
	\begin{equation}\label{eq:iter}
		\begin{split}
			u^{(0)}(t,x)=\frac{1}{2}(u_{0}(x+t)+u_{0}(x-t))+\frac{1}{2}\int_{x-t}^{x+t}v_{0}(y)dy+\frac{1}{2}{\vartheta}(t,x),\\
			u^{(n+1)}(t,x)=u^{(0)}(t,x)+\frac{1}{2}\int_{0}^{t}ds\int_{x-(t-s)}^{x+(t-s)}f(s,y,u^{(n)}(s,y))dy,\quad n\ge 0.
		\end{split}
	\end{equation}
	Taking the random constants from A\ref{asswxu}, A\ref{asswxv} and~\eqref{eq:vbound}, denote
	\[
	M(\omega) = C_{u_0}(\omega) + T C_{v_0}(\omega) + \dfrac{1}{2}C_{{\vartheta}}(\omega),
	\]
	then
	\[
	|u^{(0)}(t,x)|\le M(\omega)<\infty.
	\]
	Further, denote
	\begin{equation*} 
		M_{n}(t)=\sup_{x\in\mathbb{R}}|u^{(n)}(t,x)|.
	\end{equation*}
	We see that $M_{0}(t)\le M(\omega)$. Further, we estimate $M_n$ applying induction.
	\begin{eqnarray*}
		|u^{(n+1)}(t,x)| \stackrel{A\ref{asswxfb}}{\le} |u^{(0)}(t,x)|+\frac{1}{2}\int_{0}^{t}ds\int_{x-(t-s)}^{x+(t-s)}C_{f}(1+|u^{(n)}(s,y)|)dy \\
		\le M(\omega)+\frac{C_{f}}{2}\int_{0}^{t}2(t-s)(1+M_{n}(s))ds \stackrel{t-s\le T}{\le} (M(\omega)+C_{f}T^{2})+C_{f}T\int_{0}^{t}M_{n}(s)ds.
	\end{eqnarray*}
	Denote $D(\omega)=M(\omega)+C_{f}T^{2}$, $F=C_{f}T$, and get
	$$M_{n+1}(t)\le D(\omega)+F\int_{0}^{t}M_{n}(s)ds.$$
	Applying induction, we obtain
	\[
	M_{n}(t)\le D(\omega)\sum_{k=0}^{n}\frac{(Ft)^{k}}{k!},
	\]
	and get $M_n(t)\le \widetilde{M}(\omega)$, where $\widetilde{M}(\omega)=D(\omega)e^{FT}$.
	
	Further, we will show that the sequence $u^{(n)}$ converges.
	Denote 
	\[
	h_{n}(t)=\sup_{x\in\mathbb{R}}|u^{(n+1)}(t,x)-u^{(n)}(t,x)|.
	\]
	This random value is finite because $u^{(n)}$ are bounded, $h^{(n)}\le 2\widetilde{M}(\omega)$.
	For $n\ge1$, we have
	\begin{equation}\label{eq:estun}
		\begin{split}
			|u^{(n+1)}(t,x)-u^{(n)}(t,x)|\stackrel{\eqref{eq:iter}}{=} 
			\\
			\dfrac{1}{2}\Bigl|\int_{0}^{t}ds\int_{x-(t-s)}^{x+(t-s)}f(s,y,u^{(n)}(s,y))-f(s,y,u^{(n-1)}(s,y))\,dy\Bigr|\\
			\stackrel{A\ref{asswxfl}}{\le} \dfrac{1}{2}\int_{0}^{t}ds\int_{x-(t-s)}^{x+(t-s)}L_{f,\widetilde{M}(\omega)}|u^{(n)}(s,y)-u^{(n-1)}(s,y)|\,dy\\
			\Rightarrow
			h_{n}(t) \le  L_{f,\widetilde{M}(\omega)}\int_{0}^{t}(t-s)h_{n-1}(s)ds \le L_{f,\widetilde{M}(\omega)}T\int_{0}^{t}h_{n-1}(s)ds.
		\end{split}
	\end{equation}
	Taking into account that $h^{(0)}\le 2\widetilde{M}(\omega)$, by induction we obtain
	\[
	h_{n}(t)\le 2\widetilde{M}(\omega)\frac{(L_{f,\widetilde{M}(\omega)}Tt)^{n}}{n!}.
	\]
	Therefore, the series $\sum_{n=1}^\infty h_{n}(t)$ converges uniformly, and $u^{(n)}(t,x)$ converges uniformly to some function $u(t,x)$, and $|u(t,x)| \le \widetilde{M}(\omega)$. 
	Passing to the limit in~\eqref{eq:iter}, we find that  \eqref{eqwavex} holds for $u$.
	
	It remains to show that the solution is unique. Let $\hat{u}(t,x)$ be another bounded solution with the same initial conditions, and  
	\begin{equation*}
		\begin{split}
			\hat{u}(t,x)
			= \frac{1}{2}(u_0(x+t)+u_0(x-t))+ \frac{1}{2}\int_{x-t}^{x+t} v_0(y)\,dy\\
			+ \frac{1}{2}\int_0^t\,ds\int_{x-(t-s)}^{x+(t-s)}  f(s,y, \hat{u}(s,y))\,dy
			+ \frac{1}{2}\int_{|y-x|<t}\,d\mu(y) \int_{0}^{t-|y-x|}\sigma(s,y)\,ds\, .
		\end{split}
	\end{equation*}
	We take the difference with~\eqref{eqwavex}, and similarly  to~\eqref{eq:estun}, using the notations
	\[
	\hat{h}(t)=\sup_{x\in\mathbb{R}}|u(t,x)-\hat{u}(t,x)|,\quad
	\hat{M}(\omega)=\sup_{t\in T,x\in\rr}|u(t,x)|+\sup_{t\in T,x\in\rr}|\hat{u}(t,x)|
	\]
	arrive at
	\[
	\hat{h}(t) \le   L_{f,\hat{M}(\omega)}T\int_{0}^{t}\hat{h}(s)ds\quad\Rightarrow\quad\hat{h}(t)\le 2\widetilde{M}(\omega)\frac{(L_{f,\widetilde{M}(\omega)}Tt)^{n}}{n!}
	\]
	for all $n\ge 1$. Therefore, $\hat{h}=0$ and $u=\hat{u}$.
	
	2) We take $u$ and $\widetilde{\Omega}$ as defined in part 1) of the proof. Then $|u(t,x)|\le \widetilde{M}(\omega)$, and, by A\ref{asswxfb}, $|f(t,x,u)|\le C(1+\widetilde{M}(\omega))$. On the right-hand side \eqref{eqwavex}, the sum
	\[
	\frac{1}{2}\int_0^t\,ds\int_{x-(t-s)}^{x+(t-s)}  f(s,y, u(s,y))\,dy
	\]
	is Lipschitz continuous in $(t,x)$. By A\ref{asswxu}, the term
	\[
	\frac{1}{2}(u_0(x+t)+u_0(x-t))
	\]
	is H\"{o}lder continuous in $(t,x)$ with the exponent $\beta(u_0)$. The Lemma~\ref{lm:vhold} implies that $\vartheta$ is H\"{o}lder continuous in $(t,x)$ with any exponent $\gamma<\dfrac{1}{2}$.
\end{proof}

\bibliographystyle{plain}
\bibliography{Debelyak_Radchenko_WaveEq_BullKNUbib}

\end{document}